\documentclass{proc-larxiv}

\usepackage{float}
\usepackage[final]{graphicx}
\usepackage{epstopdf} 
\usepackage{caption}
\usepackage{subcaption}

\usepackage[utf8x]{inputenc}

\usepackage[T1]{fontenc}
\usepackage{mathrsfs}
\usepackage{mathtools}

\usepackage{tikz}
\usetikzlibrary{decorations.markings}
\usetikzlibrary{decorations.pathreplacing}
\usetikzlibrary{arrows,shapes,positioning,patterns}
\usetikzlibrary{knots}
\tikzstyle{none}=[inner sep=0pt]
\pgfdeclarelayer{edgelayer}
\pgfdeclarelayer{nodelayer}
\pgfsetlayers{edgelayer,nodelayer,main}
\usepackage{url} 
\usepackage{verbatim}

\newcommand{\dfn}[1]{\emph{#1}}

\newcommand{\bS}{\mathbb S}

\newcommand{\R}{\mathbb{R}}

\newcommand{\bZ}{\mathbb Z} 
\newcommand{\Z}{\mathbb{Z}}

\newcommand{\rk}{{\rm rank} \,}

\renewcommand{\epsilon}{\varepsilon}

\newcommand{\om}{\omega}

\DeclareMathOperator{\wsp}{wsp}
\DeclareMathOperator{\ssp}{sp}

\newcommand{\lk}{\ell k}

\newcommand{\cfl}{\rm CFL^{\text{--}}}

\newtheorem{thm}{Theorem}[section]
\newtheorem{lemma}[thm]{Lemma}
\newtheorem{cor}[thm]{Corollary}
\newtheorem{rem}[thm]{Remark}
\newtheorem{prop}[thm]{Proposition}

\theoremstyle{definition}
\newtheorem{defn}[thm]{Definition}

\newtheorem{exa}[thm]{Example}

\numberwithin{equation}{section}

\usepackage{xcolor}
\usepackage[colorlinks,
    linkcolor={red!50!black},
    citecolor={blue!50!black},
    urlcolor={blue!80!black}]{hyperref}

\numberwithin{equation}{section}

\begin{document}

\title{A note on the weak splitting number}

\author{Alberto Cavallo}
\address{Max Plank institute for Mathematics, Vivatsgasse 7, 53111 Bonn, Germany}
\email{cavallo@mpim-bonn.mpg.de}

\author{Carlo Collari}
\address{New York University Abu Dhabi, PO Box 129188, Saadiyat Island, Abu Dhabi, UAE}
\email{carlo.collari.math@gmail.com}

\author{Anthony Conway}
\address{Max Plank institute for Mathematics, Vivatsgasse 7, 53111 Bonn, Germany}
\email{anthonyyconway@gmail.com}
\thanks{This work started when the second-named author was visiting the first-named author in Bonn. All three authors are grateful to the Max Plank Institute in Bonn for its support and hospitality.
We thank Paolo Lisca and Chuck Livingston for helpful discussions.
We are indebted to an anonymous referee for pointing out a mistake in a previous version of this paper, and for their valuable comments.
}

\subjclass[2010]{57M27 (57M25)}

\begin{abstract} 
The weak splitting number $\wsp(L)$ of a link $L$ is the minimal number of crossing changes needed to turn $L$ into a split union of knots.
We describe conditions under which certain $\mathbb{R}$-valued link invariants give lower bounds on $\wsp(L)$.
This result is used both to obtain new bounds on $\wsp(L)$ in terms of the multivariable signature and to recover known lower bounds in terms of the $\tau$ and $s$-invariants.
We also establish new obstructions using link Floer homology and apply all these methods to compute $\wsp$ for all but two of the $130$ prime links with $9$ or fewer crossings.
\end{abstract}

\maketitle
\section{Introduction}

Given a link $L$, the \emph{weak splitting number} $\wsp(L)$ is the minimal number of crossing changes
%are necessary
needed to convert $L$ into a \emph{completely split link}, i.e. into a disjoint union of knots contained in pairwise disjoint balls.
This paper studies $\wsp(L)$ using a variety of link invariants, including signatures and the $J$-function from link Floer homology.

The weak splitting number, which was first introduced by Adams~\cite{Adams96}, must not be confused with the similarly defined \textit{splitting number} $\ssp(L)$; the definition of the latter only allows crossing changes between distinct components, which we call \dfn{mixed crossing changes}.
While the splitting number has been intensively studied \cite{BatsonSeed, BorodzikFriedlPowell, BorodzikGorsky, ChaFriedlPowell, CimasoniConwayZacharova, Livingston, Jeong}, the weak splitting number has so far attracted less attention~\cite{Adams96, BorodzikFriedlPowell, CavalloCollari, Shimizu}.  
Indeed, $\wsp(L)$ is harder to compute than $\ssp(L)$; one of the main reasons being that the isotopy type of the components of $L$ is not fixed under arbitrary crossing changes. 
We now review known methods to study $\wsp$ and describe new ones, using prime links with $9$ or fewer crossings to gauge their efficiency.

A first estimate on $\wsp$ is provided by the linking numbers: the sum of their absolute values is a lower bound.
Apart from this \emph{linking bound}, the multivariable Alexander polynomial $\Delta_L$ also gives rise to obstructions.
Indeed, Borodzik, Friedl and Powell \cite{BorodzikFriedlPowell} proved that if~$L$~is an~$\ell$-component link with~$\Delta_L\neq 0$, then
%we have
~$\ell-1 \leq \text{wsp}(L)$,
 and if equality is achieved, then $\Delta_L$ must factor~as
\begin{equation}
\label{eq:AlexanderObstruction}
\Delta_L(t_1,\ldots,t_\ell)=f \overline{f} \cdot \prod_{i=1}^\ell p_i(t_i) \cdot \prod_{i=1}^\ell (t_i-1)^{m_i}.
\end{equation}
Another method to compute $\wsp$ relies on~\emph{slice-torus link invariants}~\cite{LivingstonSliceTorus,Lewark, CavalloCollari}.
These are numerical concordance invariants that include Oszv\'ath and Szab\'o's $\tau$-invariant~\cite{OzsvathSzabo, CavalloTau}, and a normalisation of Rasmussen's $s$-invariant~\cite{Rasmussen, BeliakovaWehrli}.
More~precisely, the first two authors observed in~\cite{CavalloCollari} that if $\nu$ is a slice-torus link invariant and~$L=K_1 \cup\ldots\cup K_\ell$, then
\begin{equation}
\left\vert \nu(L) - \sum_{i=1}^{\ell} \nu(K_i) \right\vert \leq \wsp(L).
\label{eq:nu-bound}
\end{equation}
Together with the linking bound and Borodzik, Friedl and Powell's \emph{Alexander obstruction},
% from~\eqref{eq:AlexanderObstruction}, 
the \emph{slice-torus bound} in \eqref{eq:nu-bound} allow us to determine the weak splitting number of $114$ out of the $130$ prime links with $9$ or fewer crossings; see Table~\ref{tab:small_linksI}.

In order to determine the remaining cases, we develop novel lower bounds and obstructions.
Firstly, we observe that the multivariable signature $\sigma_L$ and nullity $\eta_L$ of Cimasoni-Florens~\cite{CimasoniFlorens} can be leveraged to provide lower bounds on the weak splitting number.
\begin{thm}
\label{thm:SignatureIntro}
If $L = K_1 \cup \ldots \cup K_\ell$ is an oriented link and~$\underline{\omega}=(\omega_1,\ldots,\omega_\ell)\in (\mathbb{S}^1)^\ell$, then the following inequality holds:
\begin{small}
\begin{equation*}
\left\vert \sigma_L(\underline{\omega}) - \sum_{i=1}^{\ell} \sigma_{K_i}(\omega_i) \right\vert + \left\vert \eta_L(\underline{\omega}) - \sum_{i=1}^{\ell} \eta_{K_i}(\omega_i) - \ell + 1\right\vert + 3\sum_{i<j} \vert \ell k (K_i,K_j)\vert \leq 4 \wsp(L). 
\label{eq:sigma-bound}
\end{equation*}
\end{small}
\end{thm}

While this \emph{signature bound} has appeared in the (unpublished) PhD thesis of the third named author~\cite{ConwayThesis}, here we provide an alternative proof.
To state the theorem on which this alternative proof is based, and which is one of the main results of this paper, we introduce some terminology.
A \dfn{self crossing change} is a crossing change that involves only one component.
If a link is oriented, a \emph{positive crossing change} is one that changes a negative crossing into a positive crossing.
\begin{thm}\label{prop:main_result1}
Let $I$ be an $\R$-valued oriented link invariant. Suppose there exists~$a,b,b'\in~\R$ such that
\begin{equation}
\label{eq:self}
I(L)-I(L^\prime) \in [a,b]\quad \text{or} \quad I(L)-I(L^\prime) \in [-b',b']
\end{equation}
depending on whether $L$ and $L'$ are related by a positive self crossing change or a mixed crossing change. 
If $\delta := (b-a)-b^\prime \geq 0$, then for each oriented link $L=K_1 \cup \ldots \cup K_\ell$
%, we have
 \begin{equation}
\label{eq:SkeinBound}
 \left\vert I(L) - I\left(\bigsqcup_{i=1}^\ell K_{i}\right) \right\vert + \delta \sum_{i=1}^{\ell} \vert \lk(K_i,K_j) \vert \leq (b-a) \wsp(L).
\end{equation}
\end{thm}

Theorem~\ref{prop:main_result1} thus provides a template to produce lower bounds on $\wsp$.
As applications, we recover the slice-torus and signature bounds; see Corollary~\ref{cor:Jfunction} for a third application.

Before returning to links with $9$ or fewer crossings, we pause and compare the lower bounds that we have obtained so far. 
One might expect the
% modern 
obstructions from link homology theories to be more powerful than classical invariants. The next proposition shows that this is not always the case (cf.~Propositions~\ref{prop:sngeqst} and~\ref{prop:stgeqsn}, and Remark~\ref{rem:linking is better}), answering a question posed in \cite[Remark~1.5]{CavalloCollari}.

\begin{prop}
The linking bound, the slice-torus bound, and the signature bound are independent.  
More precisely, for each of the above bounds there are infinitely many links for which the given bound is sharper than the other two. Moreover, the difference between the values of any two among the above-mentioned bounds can be arbitrarily high.
\end{prop}

We now return to weak splitting numbers of links with 9 or fewer crossings.
Using the signature bound from Theorem~\ref{thm:SignatureIntro}, we are able to determine 6 of the missing values in Table~\ref{tab:small_linksI}.
The remainder of this article develops methods to investigate the $10$ remaining~cases.

Inspired by \cite[Theorem 7.7]{BorodzikGorsky}, we first develop new obstructions based on the $J$-function from link Floer homology~\cite{GorskyNemethi,BorodzikGorsky}.
The definition and properties of the {$J$-function} are reviewed in Section~\ref{sub:HF}.~For the moment we only note that the $J$-function
\[ J_L \colon  \bZ^{\ell} \longrightarrow \Z_{\geq 0} \]
is an invariant of the $\ell$-component oriented link $L$. 
To state our result we also need the integer-valued knot concordance invariant $\nu^+$ introduced by Hom and Z.~Wu~\cite{HomWu}. 
In fact,~$\nu^+(K)$ can be defined as the minimal $m\in\Z_{\geq 0}$ such that $J_K(m)=0$; see~Section~\ref{sub:HF}~as well as~\cite[Definition 2.12 and Proposition 2.13]{OzsvathStipsiczSzabo}.

Finally, given a splitting sequence for an oriented link $L = K_1 \cup \ldots\cup K_\ell$, we use $s_i$ (resp. $m_{i,j}^+$) to denote the number of self crossing changes performed on $K_i$ (resp. the number of positive mixed crossing changes involving both $K_i$ and $K_j$).
%Our result, which is a consequence of the more general Theorem \ref{thm:newobstructionHFLDetailed}, is the following. 

\begin{thm}
\label{thm:newobstructionHFL}
Let $L=K_1\cup \ldots \cup K_\ell $ be an oriented link, and let $\lbrace \varepsilon_{i,j} \rbrace_{i \neq j} \subset \lbrace 0,1\rbrace^{\ell^2-\ell}$ be a sequence of $\ell^2-\ell$ integers with $\varepsilon_{i,j}+\varepsilon_{j,i}=1$.
If $J_L(v_1,\ldots,v_\ell) \neq 0$ then, for some~$i$, 
%we have
\begin{equation}
\label{eq:Inequalityvi}
v_i <  \nu^{+}(K_i) + s_i + \sum_{j \neq i} \varepsilon_{i,j} m^+_{i,j}.
\end{equation}
\end{thm}

%\begin{thm}
%\label{thm:newobstructionHFL}
%Let $L=K_1\cup \ldots \cup K_\ell $ be an oriented link. If $J_L(v_1,\ldots,v_\ell) \neq 0$, then
%$$ \sum_{i=1}^\ell v_i -\sum_{i=1}^\ell \nu^+(K_i)  < \wsp(L).$$
%\end{thm}

This new obstruction still does not allow us to determine the missing values of Table~\ref{tab:small_linksI}. 
Nevertheless, in Example \ref{ex:Jfunct} we describe an infinite family of links for which the $J$-function determines $\wsp$, whereas the linking and signature bounds are ineffective.

To conclude the computation of $\wsp$ 
for 
%of
the links in Table \ref{tab:small_linksI}, Section~\ref{sec:coveringcalculus} uses homotopical considerations as well as covering link calculus.
Here, recall that for an $\ell$-component link~$L=K_1 \cup \ldots \cup K_\ell$ with $K_i$ unknotted, one can form the $2$-fold cover $p \colon \mathbb{S}^3 \to \mathbb{S}^3$ branched along $K_i$.
The link $\widetilde{L}=p^{-1}(L \setminus K_i)\subset \bS^3$ is called the \textit{covering link} of $L$.

\begin{prop}
\label{prop:HomotopyIntro}
Let $L=K_1 \cup \ldots \cup K_\ell$ be an $\ell$-component link.
\begin{enumerate}
\item If~$L$ can be split via~$k$ crossing changes that do not involve an unknotted component~$K_i$, then the corresponding covering link satisfies~$ \wsp(\widetilde{L}) \leq 2k$.
\item If $L$ can be split via self crossing changes that do not involve $K_1$, then $L\setminus K_1$ is null-homotopic in the exterior of $K_1$. 
\item
% If $\ell\geq 3$,
If $L$ has pairwise vanishing linking numbers and is not null-homotopic, then either $\wsp(L)=\operatorname{sp}(L)$ or $3 \leq \wsp(L) \leq \operatorname{sp}(L)-1$. 
\end{enumerate}
\end{prop}

Combining Proposition~\ref{prop:HomotopyIntro} with the previously described methods, we are able to determine the weak splitting numbers of all but two of the prime links with $9$ or fewer crossings.
All our lower bounds and obstructions failed to determine the weak splitting number for the links $L9a29$ and $L9a30$ in Thistlethwaite's link table.
 
\subsection*{Organisation}
In Section~\ref{sec:LowerBounds}, we establish Theorem~\ref{prop:main_result1} and Theorem~\ref{thm:SignatureIntro}.
In Section~\ref{sub:HF}, we review the J-function and prove (a generalisation of) Theorem~\ref{thm:newobstructionHFL}.
In Section~\ref{sec:Homotopy}, we prove the homotopical obstructions of Proposition~\ref{prop:HomotopyIntro}, while Section~\ref{sec:Table} lists the weak splitting numbers of all but two of the 130 links with 9 or fewer crossings.

\section{Lower bounds}\label{sec:general_lower_bounds}
\label{sec:LowerBounds}

\subsection{Linking numbers}

This subsection shows that the linking numbers as well as the number of ``obstructive sublinks'' provide a lower bound on the weak splitting number.
\medbreak

%Given a link $L$, a non-split multi-component sublink $J\subseteq L$ is called \dfn{obstructive} if it has vanishing linking matrix. 
Given a link $L$, a non-split sublink $J\subseteq L$ is called \dfn{obstructive} if it has vanishing linking matrix. 
A collection of sublinks $J_1,\ldots, J_k\subseteq L$ is called an \dfn{obstructive collection} if each $J_i$ is obstructive, and if the $J_i$'s are pairwise disjoint (i.e. they do not have common components).
We use $n_{oc}$ to denote the maximal number of elements among all obstructive collections of sublinks of $L$.

The next result shows that linking numbers and the number of obstructive links provide lower bounds on the weak splitting number. The proof is identical to that of~\cite[Lemma~$2.1$]{ChaFriedlPowell}, with a small \emph{caveat}: there exists links with an arbitrary number of components, which have pairwise linking number $0$, and $\wsp$ equal to $1$.

\begin{lemma}
 \label{lem:pre_linking_bound}
If a splitting sequence for an $\ell$-component link $L=K_1 \cup \ldots \cup K_\ell$ has~$s$ self crossing changes and $m$ mixed crossing changes, then
\[\sum_{i<j} \vert \ell k(K_i,K_j) \vert  \leq m\qquad\text{and}\qquad m \equiv \sum_{i<j}  \ell k(K_i,K_j) \: \mod\,2.\]
Additionally, the linking numbers give a lower bound on the weak splitting number:
\begin{equation}
 \sum_{i<j} \vert \ell k(K_i,K_j) \vert +n_{oc}(L) \leq\wsp(L).
 \label{eq:linking_bound_simple}
 \end{equation}
%Furthermore, the equality holds only if $ \ssp(L) = \wsp(L)$. 
\end{lemma}

The following example shows how Lemma~\ref{lem:pre_linking_bound} can be applied in practice.

\begin{exa}
The $3$-component link $L9a47$ has $ \wsp(L9a47)= 3$.
Indeed, the inspection of a diagram shows that $ \wsp(L9a47) \leq 3$.
The equality follows from Lemma~\ref{lem:pre_linking_bound}, since we have~$\sum_{i<j} \left| \ell k(K_i,K_j)\right| = 2$, and $L9a47$ contains the Whitehead link as a sublink.
\end{exa}

We observe that linking numbers provide a condition for the equality $\ssp(L)=\wsp(L)$.

\begin{rem}
We argue that if $N:=\sum_{i<j} \ell k(K_i,K_j)=\wsp(L)$, then  $\wsp(L)=\ssp(L)$. 
We need only show that $\ssp(L) \leq \wsp(L)$.
Choose a minimal splitting sequence with~$s$ self crossing changes and $m$ mixed crossing changes. 
By Lemma~\ref{lem:pre_linking_bound} we have~$N\leq m \leq s+m=N$. It follows that $s=0$ and $\ssp(L) \leq m=\wsp(L)$.
%In particular, $N=\wsp(L)=s+m$.
%If $s>0$, then by Lemma~\ref{lem:pre_linking_bound} we have~$N\leq m <s+m=N$, which is absurd. Therefore, $s=0$ and $\ssp(L) \leq m=\wsp(L)$.
%%Note : this argument shows that if~\eqref{eq:linking_bound_simple} is an equality, then  any minimal splitting sequence features only mixed crossing changes. 
\end{rem}

\subsection{General Bounds}
We prove Theorem~\ref{prop:main_result1}:
%from the introduction:
any $\R$-valued oriented link invariant that has a bounded behaviour with respect to crossing changes (recall~\eqref{eq:self} in the statement of  Theorem \ref{prop:main_result1}) 
provides
%gives 
a lower bound on the weak splitting number.

\begin{proof}[Proof of Theorem \ref{prop:main_result1}]
Fix a minimal splitting sequence~$L=L^{(0)},L^{(1)},\ldots,L^{(k)}$.~%
Denote by $s$ the number of self crossing changes in said splitting sequence, and denote by $m$ the number of mixed crossing changes.
In particular, we have~$ \wsp(L) = k = s + m$.~%
Given an oriented link $J=J_1 \cup \ldots \cup J_\ell$, we consider the difference
 \[ i(J) := I(J) - I\left(\bigsqcup J_{i}\right),\]
and we wish to study the behaviour of $i(L^{(j)}) - i(L^{(j-1)})$. First, when $L^{(j)}$ is obtained from~$L^{(j-1)}$ by a self crossing change, we can apply~\eqref{eq:self} to deduce that
\begin{equation}
\label{eq:SelfFirst}
\begin{cases}
-(b-a) \leq I(L^{(j)}) - I(L^{(j-1)}) - b \leq 0 & \text{if the crossing is positive}, \\
 0 \leq I(L^{(j)}) - I(L^{(j-1)}) + b \leq (b-a) &\text{if the crossing is negative.} 
\end{cases}
\end{equation}
%depending on whether the crossing changed is positive or negative.
For each $r$, use $K^{(r)}_{1},\, \ldots,\, K^{(r)}_{\ell}$ to denote the components of $L^{(r)}$.~%
Consider the links~$\sqcup_r K^{(j)}_{r}$ and~$\sqcup_r K^{(j-1)}_{r}$ obtained as the split unions of the components of $L^{(j)}$ and $L^{(j-1)}$, respectively.
These links differ by a self crossing change, which is of the same type as the crossing change performed to pass from $L^{(j-1)}$ to $L^{(j)}$.
Thus, a second application of~\eqref{eq:self} gives the following inequalities:
\begin{equation}
\label{eq:SelfSecond}
\begin{small}
\begin{cases}
0 \leq -I\left(\bigsqcup_{r} K^{(j)}_{r}\right) + I\left(\bigsqcup_{r} K^{(j-1)}_{r}\right) + b \leq (b-a) & \text{if the crossing is positive}, \\
-(b-a) \leq -I\left(\bigsqcup_{r} K^{(j)}_{r}\right) + I\left(\bigsqcup_{r} K^{(j-1)}_{r}\right) -b \leq 0 & \text{if the crossing is negative}.
\end{cases}
\end{small}
\end{equation}
%depending on whether the crossing changed is positive or negative.
Adding the inequalities in~\eqref{eq:SelfFirst} to those in~\eqref{eq:SelfSecond}, we obtain (regardless of the type of the crossing) the inequality
\begin{equation}
%\tag{\text{\xpossame}}
\label{eq:inter}
-(b-a) \leq i(L^{(j)}) - i(L^{(j-1)}) \leq (b-a).
\end{equation}
Now, assume that the crossing change between $L^{(j-1)}$ and $L^{(j)}$ involves two different components.
A similar reasoning to the one above yields the following inequality:
\begin{equation}
%\tag{\xnegdiff}
\label{eq:extra}
-b^\prime \leq i(L^{(j)}) - i(L^{(j-1)}) \leq b^\prime.
\end{equation}
%since the components of $L^{(i)}$ and $L^{(i-1)}$ are the same.
Recall that $m$ (resp. $s$) denotes the number of mixed (resp. self) crossing changes in our fixed splitting sequence for $L$.
We have that for $s$ indices $j_1, \ldots, j_s$ Equation \eqref{eq:extra} holds, while for the remaining $m$ indices Equation \eqref{eq:inter} holds. Adding all these equations, and taking into account that $i(L^{(s+m)}) = 0$, we get
$$\vert i(L) \vert \leq (b-a)s + b^\prime m \  =(b-a)(s+m)+(b^\prime-(b-a)) m.$$
Now, since $0 \leq (b-a)-b^\prime  \  = \delta$, Lemma~\ref{lem:pre_linking_bound} implies that
\[-\delta m \leq -\delta \sum_{i<j} \vert \ell k(K_i,K_j) \vert  \]%
and the result is an immediate consequence of the following computation:%
\[\vert i(L) \vert  \leq (b-a)(s+m)-\delta m \leq (b-a)\wsp(L)-\delta \sum_{i<j} \vert \ell k(K_i,K_j) \vert . \qedhere \]
\end{proof}

We can now recover~\cite[Theorem 1.4]{CavalloCollari} from Theorem \ref{prop:main_result1}.
In particular, it follows from~\cite[Examples 2.1, 2.2, and 2.3]{CavalloCollari} that the $s$, $\tau$ and $s_n$-invariants (i.e.~the $\mathfrak{sl}_n$-analogues of $s$ \cite{Lobb, Wu}) all give rise to lower bounds for $\wsp$. The reader is referred to \cite{CavalloCollari} for the definition and general properties of slice-torus link invariants.

\begin{cor}[Slice-torus bound]\label{cor:ACbound}
Let $\nu$ be a slice-torus link invariant. If $L$ is an $\ell$-component oriented link, then
\[\left\vert \nu(L) - \sum_{i=1}^\ell \nu(K_{i}) \right\vert \leq \wsp(L).\]
\end{cor}
\begin{proof}
Slice-torus link invariants are known to satisfy the hypotheses of Theorem \ref{prop:main_result1} with $a = 0$ and~$b = b^\prime = 1$; see~\cite[Proposition 2.9]{CavalloCollari}.
Since slice-torus link invariants are, by definition, additive under disjoint unions, the corollary follows.
\end{proof}

Theorem~\ref{prop:main_result1} can be used to obtain a lower bound for~$\wsp$~from (finite) families of invariants which are uniformly bounded with respect to crossing changes, in the sense of~\eqref{eq:FamilyBound}.
%\eqref{eq:uniformly_bounded_self} and~
%\eqref{eq:uniformly_bounded_mixed}.~%
We note that the bound obtained in the following proposition is stronger than the bound obtained by applying na\"{i}vely Theorem~\ref{prop:main_result1}
to the sum of the invariants.

\begin{prop}\label{prop:uniformly_bounded}
Let $\{ I_{1},\,\ldots,\,I_{k} \}$ be a family of $\R$-valued oriented link invariants.
Suppose there exists $\Delta,\, \beta \in~\mathbb{R}$ such that
\begin{equation}
\label{eq:FamilyBound} 
\sum_{j=1}^k \left\vert I_{j}( L) - I_{j}(L^\prime) \right\vert\leq \Delta 
\quad \text{ or } \quad
\sum_{j=1}^k \left\vert I_{j}( L) - I_{j}(L^\prime) \right\vert\leq \beta
\end{equation}
depending on whether $L$ and $L'$ are related by a self crossing change or a mixed crossing change.
 If $\delta := 2 \Delta - \beta \geq 0$, then for each oriented link~$L = K_{1} \cup \ldots \cup K_{\ell}$, we have
\begin{equation}
\label{eq:SumOfInvariants}
\sum_{j=1}^{k} \left\vert I_{j}(L) - I_{j}\left(\bigsqcup_{i=1}^\ell K_{i}\right) \right\vert + \delta \sum_{i<j} \vert \ell k(K_{i},K_{j})\vert\leq 2 \Delta \wsp(L).
\end{equation}
\end{prop}
\begin{proof}
Fix $\underline{\epsilon} = (\epsilon_1 , \ldots ,\epsilon_{k}) \in \{ \pm 1 \}^{k}$, and consider the sum~$ I(\underline{\epsilon}) = \sum_{j=1}^{k} \epsilon_{j}I_{j}$.~%
Using~\eqref{eq:FamilyBound},
%Equations~\eqref{eq:uniformly_bounded_self} and~\eqref{eq:uniformly_bounded_mixed}, 
a quick verification shows that $I(\underline{\epsilon})$ satisfies the hypothesis of Theorem~\ref{prop:main_result1} with~$-a = b = \Delta$ and $b^\prime = \beta$, for each choice of~$\underline{\epsilon}$. 
Applying Theorem \ref{prop:main_result1}, we deduce that the following inequality holds for every $\underline{\epsilon}$ and every $L$:
\begin{equation}
\label{eq:Intermediate}
 \left\vert I(\underline{\epsilon})(L) - I(\underline{\epsilon})\left(\bigsqcup_{i=1}^\ell K_{i}\right) \right\vert + \delta \sum_{i<j} \vert \ell k(K_{i},K_{j})\vert\leq 2 \Delta \wsp(L).
 \end{equation}
To conclude, it remains to arrange the position of the absolute values; compare~\eqref{eq:Intermediate} with~\eqref{eq:SumOfInvariants}.
To achieve this, fix an arbitrary link $L= K_{1} \cup \ldots \cup K_{\ell}$, and choose any sequence $\underline{\epsilon}$ of signs so that
\[\epsilon_{j}\left( I_{j}(L) - I_{j}\left(\bigsqcup_{i=1}^\ell K_{i}\right) \right) = \left\vert I_{j}(L) - I_{j}\left(\bigsqcup_{i=1}^\ell K_{i}\right) \right\vert,\quad \text{ for all }  j\in\{ 1,\ldots,k\}.\]
%The statement follows from the arbitrary choice of $L$.
Since such a choice can be performed for each $L$, the proof of the proposition is concluded.
\end{proof}

As an application of Proposition~\ref{prop:uniformly_bounded} we (re-)obtain the lower bounds on $\wsp$ that appeared in the third author's (unpublished) PhD~thesis~\cite[Proposition 4.4.5]{ConwayThesis}.

We briefly recall the definition of the multivariable signature and nullity, referring to~\cite{CimasoniFlorens} for details. 
A \emph{C-complex} for an ordered link~$L=K_1 \cup \ldots \cup K_\ell$ consists of a
collection $F$ of Seifert surfaces~$F_1, \ldots , F_\ell$ for the components~$K_1, \ldots , K_\ell$ that intersect only along clasps.
Given a C-complex and a
sequence~$\varepsilon=(\varepsilon_1,\ldots, \varepsilon_\ell)$ of $\pm 1$'s, there are $2^\ell$ \emph{generalized Seifert matrices}~$A^\varepsilon$, which extend the usual Seifert matrix.
Note that for all~$\varepsilon$, we have~$A^{-\varepsilon} = (A^\varepsilon)^T$.
Using this fact, one can check that for any~$\omega = (\omega_1,\dots,\omega_\ell)\in(\mathbb{S}^1)^\ell$, the following matrix is Hermitian:
$$H(\omega)=\sum_\varepsilon\prod_{i=1}^\ell(1-\overline{\omega}_i^{\varepsilon_i})\,A^\varepsilon.$$
Since $H(\omega)$ vanishes 
%when 
as soon as one of the coordinates of~$\omega$ is equal to~$1$, it is convenient to restrict our attention to~$\omega \in \mathbb{T}^\ell_*:= (\mathbb{S}^1 \setminus \lbrace 1 \rbrace)^\ell$.
We use $\beta_0(F)$ to denote the number of connected components of a C-complex $F$.

\begin{defn}
\label{def:Signature}
The \emph{multivariable signature and nullity} of an~ordered link~$L$ at $\omega \in \mathbb{T}^\ell_*$ are 
\begin{align*}
\sigma_L(\omega):=\operatorname{sign} H(\omega); \quad \quad 
\eta_L(\omega):=  \operatorname{null}H(\omega) + \beta_0(F)-1.
\end{align*}
\end{defn}

The multivariable signature and nullity are known not to depend on the choice of the C-complex~\cite[Theorem 2.1]{CimasoniFlorens}.
Note that the signature is \emph{not} a slice-torus invariant: even though it satisfies the first three axioms of~\cite[Definition 2]{CavalloCollari}, it fails to satisfy the fourth.
Nonetheless, we can use Proposition~\ref{prop:uniformly_bounded} to sidestep this issue and  to establish that $\sigma_L$ and~$\eta_L$ provide lower bounds on the weak splitting number.

\begin{proof}[Proof of Theorem \ref{thm:SignatureIntro}]
The invariants $I_{1}(L) = \sigma_{L}(\omega)$ and $I_{2}(L)= \eta_{L}(\omega)$ satisfy the hypotheses of Proposition~\ref{prop:uniformly_bounded} with $\Delta = 2$ and $b^\prime = 1$; see~\cite[proof of Theorem~3.1]{CimasoniConwayZacharova},~\cite[Lemma~6.2]{BenardConway}, as well as~\cite[Section~5]{CimasoniFlorens}.
Furthermore, the multivariable signature and nullity behave as follows under disjoint unions (cf.~\cite[Proposition~2.13]{CimasoniFlorens}):
\[\sigma_{\bigsqcup_{i=1}^\ell K_{i}} (\omega) = \sum_{i} \sigma_{K_{i}}(\omega_{i}),\quad \eta_{\bigsqcup_{i=1}^\ell K_{i}} (\omega) =\sum_{i} \eta_{K_{i}}(\omega_{i}) + \ell - 1.\]
By Proposition \ref{prop:uniformly_bounded}, the announced inequality is established.
\end{proof}

\subsection{Comparing the slice-torus and the signature bounds}\label{sec:comparison}

In this subsection we compare the slice-torus bound and the signature bounds, and prove their independence. 

\begin{figure}[htbp]
\centering
\hspace{-.25\textwidth}
\begin{subfigure}[htbp]{0.35\textwidth}
\centering
\begin{tikzpicture}[thick,scale = .3, rotate = 90]
%\draw[gray, opacity =.25, step = 0.0625] (-2,-10) grid (2,10);
%\draw[gray, opacity =.5, step = 0.5] (-2,-10) grid (2,10);
%\draw[white] (-.1,-6) -- (.1,-6);
\draw[red] (-.5,9) .. controls +(0,-.25) and +(0,.5) .. (.5,8);
\pgfsetlinewidth{10*\pgflinewidth}
\draw[white] (.5,9) .. controls +(0,-.25) and +(0,.5) .. (-.5,8);
\pgfsetlinewidth{.1*\pgflinewidth}
\draw (.5,9) .. controls  +(0,-.25) and +(0,.5) .. (-.5,8);

\draw (-.5,8) .. controls +(0,-.5) and  +(0,.25).. (.5,7);
\pgfsetlinewidth{10*\pgflinewidth}
\draw[white](.5,8) .. controls  +(0,-.5) and  +(0,.25)  .. (-.5,7);
\pgfsetlinewidth{.1*\pgflinewidth}
\draw[red] (.5,8) .. controls  +(0,-.5) and  +(0,.25)  .. (-.5,7);

\node at (0,6) {$t$};
\draw (-.75,5) rectangle (.75,7);
\draw[->] (.5,5) -- (.5,2);

\draw[red] (-.5,5)  .. controls  +(0,-1) and  +(0,1)  .. (-4,2);
\pgfsetlinewidth{10*\pgflinewidth}
\draw[white] (-.5,2)  .. controls  +(0,1) and  +(0,-1)  .. (-4,5);
\pgfsetlinewidth{.1*\pgflinewidth}
\draw[red,->] (-.5,2)  .. controls  +(0,1) and  +(0,-1)  .. (-4,5);

\draw[red] (-.5,0) .. controls +(0,-.25) and +(0,.5) .. (.5,-1);
\pgfsetlinewidth{10*\pgflinewidth}
\draw[white] (.5,0) .. controls +(0,-.25) and +(0,.5) .. (-.5,-1);
\pgfsetlinewidth{.1*\pgflinewidth}
\draw (.5,0) .. controls  +(0,-.25) and +(0,.5) .. (-.5,-1);

\draw (.5,-2) .. controls  +(0,.25) and +(0,-.5) .. (-.5,-1);
\pgfsetlinewidth{10*\pgflinewidth}
\draw[white]  (-.5,-2) .. controls +(0,.25) and +(0,-.5) .. (.5,-1);
\pgfsetlinewidth{.1*\pgflinewidth}
\draw[red] (-.5,-2) .. controls +(0,.25) and +(0,-.5) .. (.5,-1);

\draw[red]  (-.5,-2)   .. controls +(0,-2) and +(0,-2) ..  (-4,-2) ;
\draw[red]  (-.5,9)   .. controls +(0,2) and +(0,2) ..  (-4,9) ;

\draw  (.5,-2)   .. controls +(0,-2) and +(0,-2) ..  (4,-2) ;
\draw   (.5,9)   .. controls +(0,2) and +(0,2) ..  (4,9) ;
\draw (4,9) -- (4,-2);

\draw[red] (-4,9) -- (-4,5);
\draw[red] (-4,-2) -- (-4,2);

\draw (-.75,0) rectangle (.75,2);
\node at (0,1) {$t$};
\end{tikzpicture}
%\caption{}
\end{subfigure}
\hspace{.05\textwidth}
\begin{subfigure}[htbp]{0.35\textwidth}
\centering
\begin{tikzpicture}[thick, scale = .3]
%\draw[gray, opacity =.25, step = 0.0625] (-2,-10) grid (2,10);
%\draw[gray, opacity =.5, step = 0.5] (-2,-10) grid (2,10);
\begin{scope}[shift = {(0,5.5)}, rotate = 90]
\draw[orange,fill] (-.5,5)  .. controls  +(0,-1) and  +(0,1)  .. (-4,2) -- (-4,-2)    .. controls +(0,-2) and +(0,-2) .. (-.5,-2) -- (-.5,2) .. controls  +(0,1) and  +(0,-1)  .. (-4,5) -- (-4,9) .. controls +(0,2) and +(0,2) ..  (-.5,9);

\draw[gray,fill]  (.5,-2)   .. controls +(0,-2) and +(0,-2) ..  (4,-2) --  (4,9)   .. controls +(0,2) and +(0,2) ..  (.5,9)  ;

\draw[orange, fill] (-.5,9) .. controls +(0,-.25) and +(0,.5) .. (.5,8) -- (-.5,8);
\draw[red] (-.5,9) .. controls +(0,-.25) and +(0,.5) .. (.5,8);
\pgfsetlinewidth{2*\pgflinewidth}
\draw[white] (.5,9) .. controls +(0,-.25) and +(0,.5) .. (-.5,8);
\pgfsetlinewidth{.5*\pgflinewidth}
\draw[fill,gray] (.5,9) .. controls  +(0,-.25) and +(0,.5) .. (-.5,8) -- (.5,8);
\draw (.5,9) .. controls  +(0,-.25) and +(0,.5) .. (-.5,8);

\draw[fill,gray](.5,8)-- (-.5,8) .. controls +(0,-.5) and  +(0,.25).. (.5,7);
\draw (-.5,8) .. controls +(0,-.5) and  +(0,.25).. (.5,7);
\pgfsetlinewidth{2*\pgflinewidth}
\draw[white](.5,8) .. controls  +(0,-.5) and  +(0,.25)  .. (-.5,7);
\pgfsetlinewidth{.5*\pgflinewidth}
\draw[orange,fill] (-.5,8)-- (.5,8) .. controls  +(0,-.5) and  +(0,.25)  .. (-.5,7);
\draw[red]  (.5,8) .. controls  +(0,-.5) and  +(0,.25)  .. (-.5,7);

\draw (.5,5) -- (.5,2);

\draw[red] (-.5,5)  .. controls  +(0,-1) and  +(0,1)  .. (-4,2);
\pgfsetlinewidth{2*\pgflinewidth}
\draw[white] (-.5,2)  .. controls  +(0,1) and  +(0,-1)  .. (-4,5);
\pgfsetlinewidth{.5*\pgflinewidth}
\draw[red] (-.5,2)  .. controls  +(0,1) and  +(0,-1)  .. (-4,5);

\draw[thick, dashed] (-.5,8) -- (.5,8);

\draw[orange,fill] (-.5,0) .. controls +(0,-.25) and +(0,.5) .. (.5,-1) -- (-.5,-1)--cycle;
\draw[red] (-.5,0) .. controls +(0,-.25) and +(0,.5) .. (.5,-1);
\pgfsetlinewidth{2*\pgflinewidth}
\draw[white] (.5,0) .. controls +(0,-.25) and +(0,.5) .. (-.5,-1);
\pgfsetlinewidth{.5*\pgflinewidth}
\draw[gray,fill] (.5,0) .. controls  +(0,-.25) and +(0,.5) .. (-.5,-1) -- (.5,-1)--cycle;
\draw (.5,0) .. controls  +(0,-.25) and +(0,.5) .. (-.5,-1);

\draw[gray,fill]  (.5,-2) .. controls  +(0,.25) and +(0,-.5) .. (-.5,-1) -- (.5,-1)--cycle;
\draw (.5,-2) .. controls  +(0,.25) and +(0,-.5) .. (-.5,-1);
\pgfsetlinewidth{2*\pgflinewidth}
\draw[white]  (-.5,-2) .. controls +(0,.25) and +(0,-.5) .. (.5,-1);
\pgfsetlinewidth{.5*\pgflinewidth}

\draw[orange,fill] (-.5,-2) .. controls +(0,.25) and +(0,-.5) .. (.5,-1)-- (-.5,-1)--cycle;
\draw[red] (-.5,-2) .. controls +(0,.25) and +(0,-.5) .. (.5,-1);

\draw[thick, dashed] (-.5,-1) -- (.5,-1);

\draw[red]  (-.5,-2)   .. controls +(0,-2) and +(0,-2) ..  (-4,-2) ;
\draw[red]  (-.5,9)   .. controls +(0,2) and +(0,2) ..  (-4,9) ;

\draw  (.5,-2)   .. controls +(0,-2) and +(0,-2) ..  (4,-2) ;
\draw   (.5,9)   .. controls +(0,2) and +(0,2) ..  (4,9) ;
\draw (4,9) -- (4,-2);

\draw[red] (-4,9) -- (-4,5);
\draw[red] (-4,-2) -- (-4,2);

\draw[fill,white] (-.75,5) rectangle (.75,7);
\node at (0,6) {$t$};
\draw[] (-.75,5) rectangle (.75,7);
\draw[fill,white] (-.75,0) rectangle (.75,2);
\node at (0,1) {$t$};
\draw (-.75,0) rectangle (.75,2);
\end{scope}

\begin{scope}[shift={+(7.7,-2)}]

\draw[gray,fill] (.5,9) rectangle (1.75,7);
\draw[orange,fill] (-.5,9) rectangle (-1.75,7);
\draw[orange, fill] (-.5,9) .. controls +(0,-.25) and +(0,.5) .. (.5,8) -- (-.5,8);
\draw[red] (-.5,9) .. controls +(0,-.25) and +(0,.5) .. (.5,8);
\pgfsetlinewidth{2*\pgflinewidth}
\draw[white] (.5,9) .. controls +(0,-.25) and +(0,.5) .. (-.5,8);
\pgfsetlinewidth{.5*\pgflinewidth}
\draw[fill,gray] (.5,9) .. controls  +(0,-.25) and +(0,.5) .. (-.5,8) -- (.5,8);
\draw (.5,9) .. controls  +(0,-.25) and +(0,.5) .. (-.5,8);

\draw[fill,gray](.5,8)-- (-.5,8) .. controls +(0,-.5) and  +(0,.25).. (.5,7);
\draw (-.5,8) .. controls +(0,-.5) and  +(0,.25).. (.5,7);
\pgfsetlinewidth{2*\pgflinewidth}
\draw[white](.5,8) .. controls  +(0,-.5) and  +(0,.25)  .. (-.5,7);
\pgfsetlinewidth{.5*\pgflinewidth}
\draw[orange,fill] (-.5,8)-- (.5,8) .. controls  +(0,-.5) and  +(0,.25)  .. (-.5,7);
\draw[red]  (.5,8) .. controls  +(0,-.5) and  +(0,.25)  .. (-.5,7);
\draw[fill, white,opacity =.3] (-1.75,9) rectangle (1.75,7);
\draw[white] (-1.5,7) arc (-180:180:1.5 and 1);
\draw[blue,->] (-1.5,7) arc (180:0:1.5 and 1);
\end{scope}

\begin{scope}[shift={+(7.7,-4)}]

\draw[gray,fill] (.5,9) rectangle (1.75,7);
\draw[orange,fill] (-.5,9) rectangle (-1.75,7);
\draw[orange, fill] (-.5,9) .. controls +(0,-.25) and +(0,.5) .. (.5,8) -- (-.5,8);
\draw[red] (-.5,9) .. controls +(0,-.25) and +(0,.5) .. (.5,8);
\pgfsetlinewidth{2*\pgflinewidth}
\draw[white] (.5,9) .. controls +(0,-.25) and +(0,.5) .. (-.5,8);
\pgfsetlinewidth{.5*\pgflinewidth}
\draw[fill,gray] (.5,9) .. controls  +(0,-.25) and +(0,.5) .. (-.5,8) -- (.5,8);
\draw (.5,9) .. controls  +(0,-.25) and +(0,.5) .. (-.5,8);

\draw[fill,gray](.5,8)-- (-.5,8) .. controls +(0,-.5) and  +(0,.25).. (.5,7);
\draw (-.5,8) .. controls +(0,-.5) and  +(0,.25).. (.5,7);
\pgfsetlinewidth{2*\pgflinewidth}
\draw[white](.5,8) .. controls  +(0,-.5) and  +(0,.25)  .. (-.5,7);
\pgfsetlinewidth{.5*\pgflinewidth}
\draw[orange,fill] (-.5,8)-- (.5,8) .. controls  +(0,-.5) and  +(0,.25)  .. (-.5,7);
\draw[red]  (.5,8) .. controls  +(0,-.5) and  +(0,.25)  .. (-.5,7);

\draw[fill, white,opacity =.3] (-1.75,9) rectangle (1.75,7);
%\draw[white] (-1.5,7) arc (-180:180:1.5 and 1);

\draw[blue] (-1.5,9) arc (-180:0:1.5 and 1);
\node[blue]  at (2.9,9) {$\bf a_i$};
\node[]  at (0,12.5) {$i$-th clasp};
\node[]  at (0,6) {$(i+1)$-th clasp};
%\draw[blue,->] (-1.5,7) arc (180:0:1.5 and 1);
\end{scope}
\end{tikzpicture}
%\caption{}
\end{subfigure}
\hfill
\caption{A diagram for the link $L_{t+1}$ (left), a C-complex bounding it (centre), and the generator associated to the $i$-th clasp for $i\neq 2t$ (right). The boxes marked with $t$ indicate the presence of $\vert t \vert$ full twists equal to those illustrated if $t>0$, and their mirror if $t<0$. 
}
\label{fig:2bridge}
\end{figure}
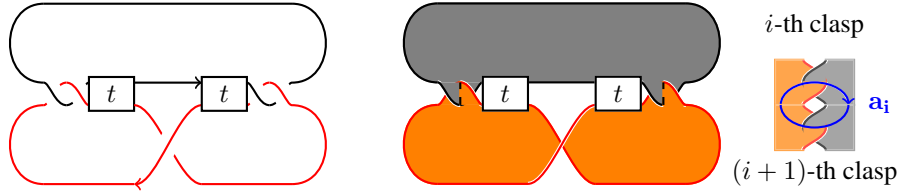

\begin{prop}\label{prop:stgeqsn}
There is an infinite family $\lbrace L_t \rbrace_{t \geq 1}$ of links for which the slice-torus bound is sharp, but for which the signature and the linking bounds are not. Furthermore, the difference between the values provided for the links  $\lbrace L_t \rbrace_{t \geq 1}$ by any two among these bounds increases linearly in $t$.
\end{prop}
\begin{proof}
Consider the diagram representing the 2-bridge link $L_t$ illustrated in Figure \ref{fig:2bridge}.
It can easily be seen that a self crossing change (on the unique crossing involving only one component) turns $L_{t+1}$ into $L_{t}$.
Thus, we deduce that
$\wsp(L_{t}) \leq t.$
Notice that the linking number of $L_t$ is zero, and therefore the linking bound from Lemma~\ref{lem:pre_linking_bound} is $1$ (and is independent of $t$).

We now use slice-torus link invariants to establish the equality $\wsp(L_{t}) = t$.
A quick computation using~\cite[Theorem 1.3]{CavalloCollari} shows that $\nu(L_{t}) \geq t$, for any slice-torus link invariant~$\nu$. 
Since the components of $L_t$ are unknots, it follows from Corollary \ref{cor:ACbound}  that 
$$\wsp(L_t) = \nu(L_{t}) = t.$$

It remains to show that the signature bound cannot be used to determine~$\wsp(L_t)$.~%
The generalised Seifert matrices corresponding to the C-complex~$F_t$ shown in Figure \ref{fig:2bridge} are of size~$\rk_\Z H_1(F_t)=2t-1$. We deduce that
\[ \vert \sigma_{L_t}(\omega) \vert + \vert \eta_{L_{t}}(\omega) -1 \vert \leq  \vert \sigma_{L_t}(\omega) \vert + \vert \eta_{L_{t}}(\omega) \vert + 1 \leq (2t -1) + 1 = 2 t.\]
Thus, the lower bound in Theorem \ref{thm:SignatureIntro} does not exceed $\lceil t/2 \rceil$, and therefore cannot be sharp for $t \geq 2$.
This concludes the proof of the proposition.
\end{proof}

Next, we construct an infinite family of links for which the signature bound is stronger than the slice-torus  and linking bounds.

\begin{prop}\label{prop:sngeqst}
There is an infinite family $\lbrace L^\prime_n \rbrace$ of links for which the signature bound is sharper than the linking and the slice-torus bounds, and their difference grows linearly with $n$.
\end{prop}
\begin{proof}
A brief inspection of the diagram representing the 3-component link $L = L10a129$, see \cite{LinkInfo}, shows that $\wsp(L) \leq 3$.
Since $L$ contains the Whitehead link as a sublink, Lemma~\ref{lem:pre_linking_bound} implies that $\wsp(L) = 3$.
Consider the link~$L^\prime_n$ obtained by connect-summing \footnote{For the definition of connected sums of links, see for instance \cite[Section 4.6]{CromwellBook}.} $n$ copies of $L$.
This connected sum can be taken so that the linking number bound fails; if we take the connected sum along two components each of which is part of a Whitehead sublink in the corresponding copy of the $L10a129$, then the number of obstructive sublinks does not increase (thus the linking bound is $2n + 1$).
Note however that the choice of the component where the connected sum is performed is immaterial for the remainder of the argument.
For instance, regardless of this choice, we have $\wsp(L^\prime_n) \leq 3n$.
 
By~\cite[Proposition~2.12]{CimasoniFlorens}, the mutivariable signature and nullity (as well as linking numbers) are additive with respect to the connected sum. Therefore, it suffices to compute the signature bound for $L$ to obtain the bound for $L_n'$.~%
Denote by $\sigma^{LT}(\omega)$ the Levine-Tristram signature, then~\cite[Proposition~2.5]{CimasoniFlorens} asserts that for any link $J$ we have
\begin{equation}
\label{eq:CFProposition25}
\sigma_J(\omega,\ldots,\omega)=\sigma_J^{LT}(\omega)+\sum_{i<k} \ell k(J_i,J_k).
\end{equation}
Using the Seifert matrices for $L10a29$ provided by LinkInfo~\cite{LinkInfo}, we see that the signature bound for~$L$ at~$\omega=e^{\pi i/4}$ is~$10/4$.~%
Using the aforementioned additivity argument, and since the number of components increases by~$2$ at each connected sum, we get
\[   \frac{10}{4}n  \leq \wsp(L_n') \leq 3n.\]
It remains to argue that the slice-torus bound is not greater than $10n/4$.
While slice-torus invariants are not additive under connected sums, they are known to satisfy the following sub-additivity property~\cite[Remark 2.8]{CavalloCollari}:
\[\nu(L_1) + \nu(L_{2}) - 1 \leq \nu (L_{1} \#_{K_{1},K_{2}} L_{2}) \leq \nu(L_1) + \nu(L_{2}).\]
On the other hand, since $L$ is non-split and alternating, we have \footnote{Alternating diagrams are homogeneous \cite{Cromwell}, and thus the bound in~\cite[Theorem 1.3]{CavalloCollari} is sharp for these diagrams~\cite[Section~5 and Remark~6.3]{Kawamura15}. In particular, the value of all slice-torus invariants coincides for non-split alternating links. As $s(L) = -\sigma(L)$ for non-split alternating links,~\eqref{eq:nu and sigma for alteranting} follows from~\cite[Example~2.2]{CavalloCollari}.} 
%that
\begin{equation}
\label{eq:nu and sigma for alteranting}
\nu(L) = \frac{-\sigma(L) + \ell -1}{2}.
\end{equation} 
Therefore, we obtain $\nu(L) \in \{ 0, 2\}$ by~\cite{LinkInfo}.~Consequently, regardless of this choice, we obtain~$ \nu(L_n') \leq n \nu(L)\leq 2n$.~%
This concludes the proof of the proposition.
\end{proof}

We conclude this section by using the examples of Proposition~\ref{prop:sngeqst} to provide examples where the linking bound is stronger than the signature bound and the slice-torus bound.
\begin{rem}\label{rem:linking is better}
As in Proposition~\ref{prop:sngeqst}, we consider the 3-component link $L:= L10a129$.
If, instead of performing the connected sums along one of the components of the Whitehead sublink of $L$ (as we did in Proposition~\ref{prop:sngeqst}), one performs the connected sum along the third component, then the resulting link $L_n'$ contains as many disjoint obstructive sublinks as connected summands. Thus, the linking bound gives $3n \leq \wsp(L^\prime_n)$, and the equality follows. The fact that the other two bounds cannot be sharp in this case (with an arbitrarily high difference) follows from the proof of Proposition~\ref{prop:sngeqst}.
%If, instead of performing the connected sums along one of the components of the Whitehead sublink, one performs the connected sum along the third component, then one gets as many disjoint obstructive sublinks as connected summands. Thus, the linking bound gives $3n \leq \wsp(L^\prime_n)$, and the equality follows. The fact that the other two bounds cannot be sharp in this case (with an arbitrarily high difference) follows from the rest of the argument below.
%
\end{rem}

\section{Bounds from Heegaard-Floer homology}
\label{sub:HF}
We prove Theorem \ref{thm:newobstructionHFL}, which provides lower bounds for $\wsp$ via link Floer homology.
First however, we briefly review the H-function of a link~\cite{GorskyNemethi, BorodzikGorsky}, an invariant that is extracted from the minus flavor $\cfl(L)$ of link Floer homology~\cite{OzsvathSzabo, Rasmussen, OzsvathSzaboLinks}.
\medbreak
Let $\mathbb{F}_2$ be the field with two elements.
Given an $\ell$-component link $L$, the complex $\cfl(L)$ is a complex of free $\mathbb{F}_2[U_1, \ldots, U_\ell]$-modules, endowed with an absolute $\bZ$-grading $d$ and a filtration for each component of $L$. 
The action of the variable $U_i$ drops the $d$-grading by 2, and each filtration level by $1$.
%the level of the filtration associated to the $i$-th component by 1. 
If we use $\underline{\lk}(L)\in \mathbb{Q}^\ell$ to denote the vector with~$\lk(K_i, L\setminus K_i)/2$ as its $i$-th entry, then the $\ell$ filtrations of $\cfl(L)$ can be re-interpreted as a unique filtration $\mathscr F$ indexed by an element of the lattice
$$\mathbb{H}(L)=\mathbb{Z}^\ell + \underline{\lk}(L).$$
In fact, there is a filtered complex $\cfl(D)$ for each Heegaard diagram $D$ of~$L$, and~$\cfl(L)$ is the filtered homotopy type, as a complex of~$\mathbb{F}_2[U_1, \ldots, U_\ell]$-modules, of \emph{any}~$\cfl(D)$~\cite{OzsvathSzaboLinks}.

As the actions of the $U_i$'s on $\cfl(L)$ are all homotopic~\cite{OzsvathSzaboLinks}, the homology of~$\cfl(L)$ can be seen as an $\mathbb F_2[U]$-module, where $U$ acts as any of the $U_i$.
It is also known that, for each $\underline m\in\mathbb{H}(L)$, the homology $H_*(\mathscr F_{\underline m}\cfl(L))$ of the $\underline m$-th  filtration level decomposes into an $\mathbb F_2[U]$-summand and an $\mathbb {F}_2[U]$-torsion summand~\cite{Surgeries}. 
The \dfn{$H$-function} of $L$ at~$\underline m\in\mathbb{H}(L)$ is then defined as
\[ H_{L}(\underline{m}) = \min \left\lbrace d \:\vert\: \rk_{\mathbb{F}[U]}\left(H_{-2d}\left(\mathscr{F}_{\underline{m}}\cfl(L)\right)\right) \ne 0 \right\rbrace.\]
This function was first introduced by Gorsky and Nemethi~\cite{GorskyNemethi}, see also~\cite{BorodzikGorsky}.
It is known that $H_L$ takes non-negative values~\cite[Proposition 3.10]{GorskyNemethi} and, as in~\cite{BorodzikGorsky}, we work with the following shifted version of $H_L$.
\begin{defn}
\label{def:J}
The \dfn{$J$-function} of an $\ell$-component link $L$ is the function
\[J_L \colon \bZ^{\ell}\longrightarrow\bZ_{\geq0} \quad  \underline m \mapsto H_L(\underline m+\underline\lk(L)).\]
\end{defn}
We use $\underline{e}_i \in \Z^\ell$ to denote the $i$-th vector of the canonical basis. We collect the properties of the $J$-function in the following proposition; proofs can be found in~\cite[Propositions 3.10 and 3.11, and Theorem 6.20]{BorodzikGorsky}.

\begin{prop}\label{prop:propertiesJ}
For an oriented link $L$, the $J$-function satisfies the following properties.
\begin{enumerate}
\item  For $i=1,\ldots, \ell$, and $\underline{v}\in\Z^\ell$, the function $J_L$ satisfies
% only takes non-negative values, and
\[ J_{L}(\underline{v}) \leq J_{L}(\underline{v} - \underline{e}_i)  \leq J_{L}(\underline{v}) + 1.\]
%\:\:\:\:\:\forall i=1, \ldots,\ell.\]
\item Let $L^\prime$ be obtained from $L$ via a positive crossing change, and let $\underline{v}\in\Z^\ell$.
\begin{enumerate}
\item  if the crossing change is a self crossing change on the $i$-th component, then
\[ J_{L^\prime}(\underline{v} + \underline{e}_i) \leq J_{L}(\underline{v}) \leq J_{L^\prime}(\underline{v});\]
\item  if the crossing change is mixed and involves the $i$-th and the $j$-th components of~$L$ then, for each $\ast\in\{i,j \}$,
\[ J_{L^\prime}(\underline{v}) \leq J_{L}(\underline{v}) \leq J_{L^\prime}(\underline{v} - \underline{e}_\ast).\]
\end{enumerate}
\item If $L=K_1 \cup \ldots \cup K_n$ is a completely split link, then
\[ J_{L}(v_1,\ldots,v_\ell) = \sum_{i=1}^{\ell} J_{K_i}(v_{i}).\]
%where $K_i$ is the $i$-th component of $L$.
\end{enumerate}
\end{prop}

Using all three items of Proposition \ref{prop:propertiesJ} and Theorem~\ref{prop:main_result1} (with $a=-1,b=0$ and $b'=1$), we obtain the following result.
% corollary.
\begin{cor}
\label{cor:Jfunction}
For each oriented link $L$, we have the following:
\[ \left\vert J_{L}(v_1,\ldots,v_\ell) - \sum_{i=1}^{\ell} J_{K_i}(v_{i}) \right\vert \leq\wsp(L).\]
\end{cor}

In order to prove Theorem~\ref{thm:newobstructionHFL} however, we need one more lemma.

\begin{lemma}
\label{lem:crossing change weak} 
Assume an $n$-component link $L$ can be split using $\wsp(L)=s+m$ crossing changes with $s$ self crossing changes and $m$ mixed crossing changes. 
Then $L$ can be converted into the split union of its components in $2s+m$ crossing changes. Furthermore, if the link $L$ is oriented, then the $2s$ self crossing changes can be taken to be $s$ positive and~$s$ negative crossing changes.
\end{lemma}
\begin{proof}
Using $\wsp(L)$ crossing changes, one can 
%convert
turn $L=K_1 \cup \ldots  \cup K_n$ into an $n$ component split link $K_1' \sqcup \ldots \sqcup K_n'$ for some knots $K_1',\ldots,K_n'$. Let $s_i$ be the number of crossing changes needed to pass from $K_i$ to $K_i'$ while splitting $L$. 
As $s=s_1+ \ldots +s_n$ and~$s_i$ is greater or equal to  the Gordian distance\footnote{This is the minimal number of crossing changes needed to pass from one given knot to another.} between $K_i$ and $K_i'$,  the link~$K_1' \sqcup \ldots \sqcup K_n'$ can be converted into~$K_1 \sqcup \ldots \sqcup K_n$ using $s$ additional crossing changes. In the case the links are oriented, then these last $s$ self crossing changes can be taken to  be of the opposite sign with respect to the $s$ self crossing changes performed on $L$. 
\end{proof}

We now prove Theorem~\ref{thm:newobstructionHFL} from the introduction.
First however, we recall some notation.
Given a splitting sequence for an oriented link $L = K_1 \cup \ldots\cup K_\ell$, we use $s_i$ (resp.~$m_{i,j}^+$) to denote the number of self crossing changes performed on $K_i$ (resp. the number of positive mixed crossing changes involving both $K_i$ and $K_j$).

%\begin{thm}
%\label{thm:newobstructionHFLDetailed}
%Let $L=K_1\cup \ldots \cup K_\ell $ be an oriented link, and let $\lbrace \varepsilon_{i,j} \rbrace_{i \purple{\neq} j} \subset \lbrace 0,1\rbrace^{\ell^2\purple{-\ell}}$ be a sequence of $\ell^2\purple{-\ell}$ integers with $\varepsilon_{i,j}+\varepsilon_{j,i}=1$.
%If $J_L(v_1,\ldots,v_\ell) \neq 0$ then, for some~$i$, 
%%we have
%\begin{equation}
%\label{eq:Inequalityvi2}
%v_i <  \nu^{+}(K_i) + s_i + \sum_{j \neq i} \varepsilon_{i,j} m^+_{i,j}.
%\end{equation}
%\end{thm}
\begin{proof}[Proof of Theorem \ref{thm:newobstructionHFL}]
We prove the contrapositive.
Use Lemma~\ref{lem:crossing change weak} to convert  $L$ into the split union of its components via $2s + m$ crossing changes, where exactly $s$ of these $2s$ self crossing changes are negative.
Applying the second item of Proposition~\ref{prop:propertiesJ}, we deduce that
%\footnote{The point of the $\varepsilon_{i,j}$ is to bound $J_{L}(v)$ by $J_{L}(v-e_*)$ instead of by $J_L(v-e_*-e_*)$.}
\begin{equation}
\label{eq:InequalityTheorem}
 J_L(\underline{v}) \leq J_{\sqcup_{i=1}^\ell K_i}\left (\underline{v}-\sum_{i=1}^\ell \left[ s_i+\sum_{j \neq i} \varepsilon_{i,j}m_{i,j}^+\right] \underline{e}_i\right).
 \end{equation}
Recall from the introduction that if $K$ is a knot and $m \geq \nu^+(K)$, then $J_K(m)=0$~\cite[Definition 2.12 and Proposition 2.13]{OzsvathStipsiczSzabo}.
 Combining this with the third item of Proposition~\ref{prop:propertiesJ}, we see that the right hand side of~\eqref{eq:InequalityTheorem} vanishes if no $v_i$ satisfies~\eqref{eq:Inequalityvi}.
 %whenever the $v_i$ do not satisfy~\eqref{eq:Inequalityvi}
The 
%first 
assertion now follows since the $J$-function is non-negative.
%The second assertion follows by taking the sum of~\eqref{eq:Inequalityvi} over $i$.
%This concludes the proof of the theorem.
\end{proof}
\begin{exa}\label{ex:Jfunct}
We use Theorem~\ref{thm:newobstructionHFL} to show that the family $L_t=K_t^1 \cup K_t^2$ of $2$-bridge links from Figure~\ref{fig:2bridge} has $\wsp(L_t)=t$. 
Since the $L_t$ are L-space links, their $J$--functions can be recovered from the potential function~\cite[Corollary 3.32]{BorodzikGorsky}. \footnote{Borodzik and Gorsky state this in terms of a symmetrized version $\Delta_L(t_1,\ldots,t_n) \in \Z[t_1^{\pm \frac{1}{2}},\ldots, t_n^{\pm \frac{1}{2}} ]$ of the Alexander polynomial for which they additionally fix a sign~\cite[Subsection 2.1]{BorodzikGorsky}. 
In other words, they are working with the potential function~$\nabla_L(t_1,\ldots,t_n)$.
Furthermore~\cite[Equation (3.3)]{BorodzikGorsky} implicitly makes use of~\cite[Theorem~1.1]{BenheddiCimasoni}
% (and~\cite[page 204]{GridHomology}) 
to obtain an equality, instead of an equality up to signs.}
Applying~\cite[Section~7.4]{BorodzikGorsky}, the potential function of~$L_t$ is
\[ \nabla_{L_t}(t_1,t_2) = (-1)^t \sum_{\vert i+\frac{1}{2}\vert + \vert j +\frac{ 1}{2} \vert \leq t} (-1)^{i+j} t_1^{i + \frac{1}{2}} t_2^{j+ \frac{1}{2}}. \]
%\[ =(-1)^t t_1^{-\frac{1}{2}}t_2^{-\frac{1}{2}}(t_1 -1)(t_2 -1)\sum_{j=0}^{t-1}\sum_{i=0}^{t-j-1}(-t_1)^{-i}(-t_2)^{-j}\left(\sum_{k=0}^{2i}t_1^k\right)\left(\sum_{h=0}^{2j}t_2^h\right)\]
If we set $\widetilde{J}_{L_t}(i,j):=J_{L_t}(i,j)-J_{K_t^1}(i)-J_{K_t^2}(j)$, then applying~\cite[Corollary 3.32]{BorodzikGorsky} and rearranging the sums of the corresponding generating function yields
\[ \widetilde{\mathbf{J}}_{L_t}(t_1,t_2) := \sum_{i,j} \widetilde{J}_{L_t}(i,j) t_1^i t_2^j =\sum_{j=0}^{t-1}\sum_{i=0}^{t-j-1}(-1)^{i+j + t + 1}\left(\sum_{k=0}^{2i}t_1^{k-i}\right)\left(\sum_{h=0}^{2j}t_2^{h-j}\right).\]
It follows from the above equalities and a tedious computation that the bound provided by Corollary \ref{cor:Jfunction} is at most $\lceil t/2\rceil$; see also~\cite[Figure 4]{Liu}.
Using successively that $L_t$ has unknotted components (as well as~$J_{\bigcirc}(v) = 0$ for $v\geq 0$, equivalently $\nu^+(\bigcirc)=0$), and the above computations, we obtain that for $r=0,\ldots, t-1$
\begin{equation}
\label{eq:ExampleHFRegion}
\widetilde{J}_{L_t}(r,t-1-r) = J_{L_t}(r,t-1-r) =1. 
\end{equation}
We already showed in Proposition~\ref{prop:stgeqsn} that~$\wsp(L)\leq t$.
By way of contradiction, assume that $\wsp(L) \leq t-1$, so that $s_1+s_2+m_{1,2}^+ \leq t-1$.
Since~$s_1 \leq t-1$, we apply~\eqref{eq:ExampleHFRegion} with~$r=s_1$ to obtain $J_{L_t}(s_1,t-1-s_1)=1$.
Since $J_{L_t}$ is non-increasing (by the first item of Proposition~\ref{prop:propertiesJ}), we deduce that $J_{L_t}(s_1,s_2+m_{1,2}^+) \geq 1$.
As $J_{L_t}(s_1,s_2+m_{1,2}^+) \neq 0$, Theorem \ref{thm:newobstructionHFL} applied to the sequence $(\varepsilon_{1,2},\varepsilon_{2,1})=(0,1)$ implies that either $s_1<s_1$ or~$s_2+m_{1,2}^+<s_2+m_{1,2}^+$.
This is a contradiction in both cases and thus $\wsp(L_t)=t$.
\end{exa}

\section{Homotopical obstructions}\label{sec:coveringcalculus}
\label{sec:Homotopy}

\subsection{Link homotopy}
\label{sub:Homotopy}

We show how the homotopy type of a link provides restrictions on its weak splitting number. 
Here, recall that two links~$L$ and~$L'$ are link-homotopic if and only if they are related by a  sequence of ambient isotopies and self crossing changes.
Furthermore, a link is nullhomotopic if it is link-homotopic to an unlink.

\begin{prop}
\label{prop:MilnorTriple}
If a link $L$ has pairwise vanishing linking numbers and is not null-homotopic, then either $\wsp(L)=\operatorname{sp}(L)$ or $3 \leq \wsp(L) \leq \operatorname{sp}(L)-1$.
\end{prop}
\begin{proof}
Write $\wsp(L)=s+m$, so that $L$ can be split using $s$ self crossing changes and~$m$ mixed crossing changes.
If one can find such a sequence with $s=0$, then there is a minimal splitting sequence without self crossing changes. Thus $\operatorname{sp}(L) \leq m=\wsp(L)$, which implies $\wsp(L)=\operatorname{sp}(L)$. % first part blue text replaces "$m$ and thus", the second replaces "in which case", and remover a $\leq \sp(L)$ to avoid redundancy.
Otherwise, every weak splitting sequence must have $s>0$ and $m$ even: indeed~$L$ has pairwise vanishing linking numbers.
%; recall Lemma~\ref{lem:pre_linking_bound}. 
Since~$\operatorname{sp}(L)$ must be even and since $m$ cannot be~$0$ (because~$L$ is not nullhomotopic), we immediately deduce that~$\wsp(L)=s+m \in \lbrace 3,\ldots, \operatorname{sp}(L)-1 \rbrace$, concluding the proof of the proposition.
\end{proof}

Since the Milnor invariants with non-repeating indices are invariant under link homotopy~\cite{Milnor}, Proposition~\ref{prop:MilnorTriple} can easily be applied in practice.
\begin{exa}
We show that the $3$-component link $L = L9a54$ has $\wsp(L) = 3$.
Since $L$ has vanishing pairwise linking numbers, the linking obstruction is ineffective, and in fact both the slice-torus bound and the signature bound give $2\leq \wsp(L) \leq 3$.

Since $L$ is known to have $\ssp(L)=4$~\cite{ChaFriedlPowell}, if we manage to show that~$L$ is not nullhomotopic, then Proposition~\ref{prop:MilnorTriple} will imply that $\wsp(L) = 3$. 
As~$L$ can be obtained from the Borromean rings $J$ via a single self crossing change, we obtain $\mu_{123}(L)=\mu_{123}(J)=1$ and thus $L$ is not nullhomotopic. 
We conclude that $ \wsp(L) = 3$, as claimed.
\end{exa}
 
The following lemma can be used to obstruct the existence of minimal weak splitting sequences without mixed crossing changes. 
\begin{prop}
\label{lem:Nullhomotopic}
Assume that the link $L$ can be completely split with only self crossing changes not involving a fixed component, say $K_1$. ~%
Then, $L\setminus K_1$ is null-homotopic in the complement of $K_1$.~%
In particular, each component of $L\setminus K_1$ is null-homotopic in $\bS^3\setminus K_1$.
\end{prop}
\begin{proof}
A self crossing change in $L \setminus K_1$ does not change its homotopy type in~$\bS^3 \setminus~K_1$. 
As any knot in a $3$-manifold that sits inside a $3$-ball is null-homotopic, the result follows.
\end{proof}

\subsection{Covering link calculus}
\label{sub:Covering} 
We use covering link calculus to study $\wsp$. 
Given an $n$-component link $L=K_1 \cup \ldots \cup K_n$ with $K_i$ unknotted, one can form the $2$-fold cover $p \colon S^3 \to S^3$ branched along $K_i$.
The link $\widetilde{L}=p^{-1}(L \setminus K_i)$ is called the \textit{covering link} of~$L$ with respect to $K_i$.
For a proof of the next result, we refer to~\cite[Section~3]{ChaFriedlPowell}.
\begin{prop}
\label{cor:Covering}
Let $L=K_1 \cup \ldots \cup K_n$ be an $n$-component link with $K_i$ unknotted.
A crossing change not involving $K_i$ results in two crossing changes in the covering link.
In particular, if $L$ can be split via $k$ crossing changes not involving $K_i$, then $ \wsp(\widetilde{L}) \leq 2k$.
\end{prop}

We show how Proposition~\ref{cor:Covering} can be used in conjunction with Proposition~\ref{lem:Nullhomotopic}: the former obstructs the existence of self crossing sequences in knotted components, while the latter obstructs the existence of self crossing changes in unknotted components.

\begin{exa}
We claim that the link~$L= K_1 \cup K_2 =L7a3$ in~Figure \ref{fig:gens pi_1} has $\wsp(L)=2$.
First, an inspection of the diagram shows that~$\wsp(L) \leq 2$, and that~$\lk(K_1,K_2) =0$.
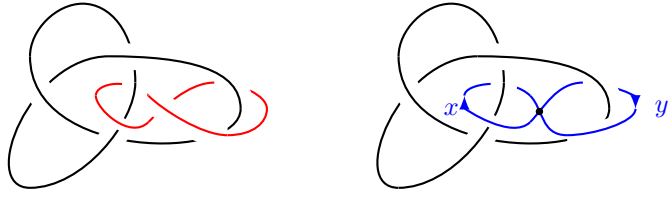
\begin{figure}[]
\centering
%\documentclass{article}
%\usepackage{tikz}
%\usepackage{ifthen}
%\usepackage{amsmath}
%\usetikzlibrary{arrows,calc,intersections}
%
%% 1) Represent n points on a circle
%% 2) Draw the complete graph on these points 
%% 3) Draw all the intersection points of any two of these segments
%\begin{document}

\begin{tikzpicture}[thick,scale = .7]

\draw (-2,-1.5) .. controls +(-1,0) and +(-1.5,0) ..(-.5,1) ;
\draw (2,0) .. controls +(0,.75) and +(1.5,0) ..(-.5,1) ;

\draw[red] (.25,.25) .. controls +(-.35,.35) and +(0,.25) .. (-.75, .25);
\draw[white, line width = 10] (-2,1) .. controls +(0,.5)  and +(-.5,0) .. (-1,2)  .. controls +(.5,0)  and +(0,1) ..(0,0.5) ;
\draw (-2,1) .. controls +(0,.5)  and +(-.5,0) .. (-1,2)  .. controls +(.5,0)  and +(0,1) ..(0,0.5) ;

\draw[red] (2.5,0) .. controls +(0,.25) and +(.5,0) .. (1.5, .5);
\draw[white, line width = 10] (2,0) .. controls +(0,.75) and +(1.5,0) ..(-.5,1) ;
\draw (2,0) .. controls +(0,.75) and +(1.5,0) ..(-.5,1) ;

\draw[white, line width = 10] (2,0) .. controls +(0,-1) and +(0,-2) ..(-2,1) ;
\draw (2,0) .. controls +(0,-1) and +(0,-2) ..(-2,1) ;

\draw[white, line width = 10] (0,0.5) .. controls +(0,-.75) and +(1,0) .. (-2,-1.5);
\draw (0,0.5) .. controls +(0,-.75) and +(1,0) .. (-2,-1.5);

\draw[red] (2.5,0) .. controls +(0,-.25) and +(.5,0) .. (1.75, -.5);

\draw[red] (.25,-.25) .. controls +(.25,.25) and +(-.5,0) .. (1.5, .5);
\draw[white, line width = 10] (.25,.25) .. controls +(.25,-.25) and +(-.5,0) .. (1.75, -.5);
\draw[red] (.25,.25) .. controls +(.25,-.25) and +(-.5,0) .. (1.75, -.5);

\draw[white, line width = 5] (.25,-.25) .. controls +(-.35,-.35) and +(0,-.25) .. (-.75, 0.25);
\draw[red] (.25,-.25) .. controls +(-.35,-.35) and +(0,-.25) .. (-.75, 0.25);

\begin{scope}[shift = {+(7,0)}]

\draw (-2,-1.5) .. controls +(-1,0) and +(-1.5,0) ..(-.5,1) ;
\draw (2,0) .. controls +(0,.75) and +(1.5,0) ..(-.5,1) ;

\draw[blue] (.5,.25) .. controls +(-.35,.35) and +(0,.25) .. (-.75, .25);
\draw[white, line width = 10] (-2,1) .. controls +(0,.5)  and +(-.5,0) .. (-1,2)  .. controls +(.5,0)  and +(0,1) ..(0,0.5) ;
\draw (-2,1) .. controls +(0,.5)  and +(-.5,0) .. (-1,2)  .. controls +(.5,0)  and +(0,1) ..(0,0.5) ;

\draw[blue,latex-] (2.5,0) .. controls +(0,.25) and +(.5,0) .. (1.5, .5);
\draw[white, line width = 10] (2,0) .. controls +(0,.75) and +(1.5,0) ..(-.5,1) ;
\draw (2,0) .. controls +(0,.75) and +(1.5,0) ..(-.5,1) ;

\draw[white, line width = 10] (2,0) .. controls +(0,-1) and +(0,-2) ..(-2,1) ;
\draw (2,0) .. controls +(0,-1) and +(0,-2) ..(-2,1) ;

\draw[white, line width = 10] (0,0.5) .. controls +(0,-.75) and +(1,0) .. (-2,-1.5);
\draw (0,0.5) .. controls +(0,-.75) and +(1,0) .. (-2,-1.5);

\draw[white, line width = 7]  (2.5,0) .. controls +(0,-.25) and +(.5,0) .. (1.2, -.5);
\draw[blue] (2.5,0) .. controls +(0,-.25) and +(.5,0) .. (1.2, -.5);

\draw[blue] (.5,-.25) .. controls +(.25,.25) and +(-.5,0) .. (1.5, .5);
\draw[blue] (.5,.25) .. controls +(.25,-.25) and +(-.5,0) .. (1.2, -.5);

\draw[white, line width = 5] (.5,-.25) .. controls +(-.35,-.35) and +(0,-.25) .. (-.75, 0.25);
\draw[blue, -latex] (.5,-.25) .. controls +(-.35,-.35) and +(0,-.25) .. (-.75, 0.25);

\draw[fill] (.675,-.05) circle (.05);
\node[blue] at (-1,0) {$x$};
\node[blue] at (3,0) {$y$};
\end{scope}
\end{tikzpicture}

%\end{document}
\caption{A diagram of the link $L7a3$ (left), and a pair of generators of the fundamental group of its trefoil component $K_1$ (right).}\label{fig:gens pi_1}
\end{figure}
Furthermore, all the techniques illustrated in Sections 2 and 3 imply that~$1 \leq \wsp(L)$.

Assume, by contradiction, that $\wsp(L) =1$.~Since $L$ has vanishing linking numbers, any minimal splitting sequence 
%must be given by
is realised by a single self crossing change.
First, we show that the self crossing change cannot occur within the trefoil component $K_1$ of $L$. Denote by $\widetilde{L}$ the lift of $K_1$ to the double cover of $\mathbb{S}^3$ branched along $K_2$, see Figure~\ref{fig:doublecover7a3}.
Since $\widetilde{L}$ has linking number~$\pm 4$, Lemma~\ref{lem:pre_linking_bound} gives $\wsp(\widetilde{L}) \geq 4$, contradicting Proposition~\ref{cor:Covering}.% according to which $\wsp(\widetilde{L}) \leq 2$. 
%any splitting sequence obtained only by self crossing changes on the trefoil component must have length at least $2$
% (and this estimate is sharp, as we can see in Figure \ref{fig:L7a3split}).
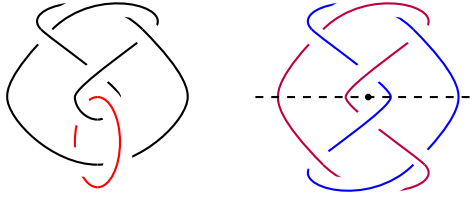
\begin{figure}[]
    \centering
    \begin{tikzpicture}[thick,scale = .3]

\draw[purple] (-2,3) .. controls +(-2, -2) and +(0,.25) .. (-4,0);
\draw[white, line width = 10]  (1,0) .. controls +(0,.5) and +(.5,-.5) .. (-2.5,3);
\draw[blue] (1,0) .. controls +(0,.5) and +(.5,-.5) .. (-2.5,3);
\draw[white, line width = 10] (-1,0) .. controls +(0,.5) and +(-.5,-.5) .. (2.5,3);
\draw[purple] (-1,0) .. controls +(0,.5) and +(-.5,-.5) .. (2.5,3);

\draw[white, line width = 10]  (-2.5,3) .. controls +(-1,1) and +(-2,2) .. (2,3);
\draw[blue]  (-2.5,3) .. controls +(-1,1) and +(-2,2) .. (2,3);

\draw[white, line width = 10]  (2.5,3) .. controls +(1,1) and +(2,2) .. (-2,3);
\draw[purple] (2.5,3) .. controls +(1,1) and +(2,2) .. (-2,3);

\draw[white, line width = 10] (2,3) .. controls +(2, -2) and +(0,.25) .. (4,0);
\draw[blue] (2,3) .. controls +(2, -2) and +(0,.25) .. (4,0);

\begin{scope}[rotate = 180]

\draw[blue] (-2,3) .. controls +(-2, -2) and +(0,.25) .. (-4,0);
\draw[white, line width = 10]  (1,0) .. controls +(0,.5) and +(.5,-.5) .. (-2.5,3);
\draw[purple] (1,0) .. controls +(0,.5) and +(.5,-.5) .. (-2.5,3);
\draw[white, line width = 10] (-1,0) .. controls +(0,.5) and +(-.5,-.5) .. (2.5,3);
\draw[blue] (-1,0) .. controls +(0,.5) and +(-.5,-.5) .. (2.5,3);

\draw[white, line width = 10]  (-2.5,3) .. controls +(-1,1) and +(-2,2) .. (2,3);
\draw[purple]  (-2.5,3) .. controls +(-1,1) and +(-2,2) .. (2,3);

\draw[white, line width = 10]  (2.5,3) .. controls +(1,1) and +(2,2) .. (-2,3);
\draw[blue] (2.5,3) .. controls +(1,1) and +(2,2) .. (-2,3);

\draw[white, line width = 10] (2,3) .. controls +(2, -2) and +(0,.25) .. (4,0);
\draw[purple] (2,3) .. controls +(2, -2) and +(0,.25) .. (4,0);

\end{scope}

\draw[dashed] (-5,0) -- (5,0);
\draw (0,0) circle (.1);
\draw[fill] (0,0) circle (.05);

\begin{scope}[shift = {+(-12,0)}]

\draw(-2,3) .. controls +(-2, -2) and +(0,.25) .. (-4,0);
\draw[white, line width = 10]  (1,0) .. controls +(0,.5) and +(.5,-.5) .. (-2.5,3);
\draw (1,0) .. controls +(0,.5) and +(.5,-.5) .. (-2.5,3);
\draw[white, line width = 10] (-1,0) .. controls +(0,.5) and +(-.5,-.5) .. (2.5,3);
\draw (-1,0) .. controls +(0,.5) and +(-.5,-.5) .. (2.5,3);

\draw[white, line width = 10]  (-2.5,3) .. controls +(-1,1) and +(-2,2) .. (2,3);
\draw (-2.5,3) .. controls +(-1,1) and +(-2,2) .. (2,3);

\draw[white, line width = 10]  (2.5,3) .. controls +(1,1) and +(2,2) .. (-2,3);
\draw (2.5,3) .. controls +(1,1) and +(2,2) .. (-2,3);

\draw[white, line width = 10] (2,3) .. controls +(2, -2) and +(0,.25) .. (4,0);
\draw (2,3) .. controls +(2, -2) and +(0,.25) .. (4,0);

\draw (0,-1) .. controls +(.5, 0) and +(0,-.5) .. (1,0);
\draw (0,-3) .. controls +(2, 0) and +(0,-1) .. (4,0);

\draw[white, line width = 10]  (0,0) arc (90:-90: 1 and 2);
%\draw[fill, opacity =.3] (0,-2) circle (1 and 2);

\draw[red] (0,0) arc (90:-90: 1 and 2);

\draw[red] (0,0) arc (90:270: 1 and 2);

\draw[white, line width = 10] (0,-1) .. controls +(-.5, 0) and +(0,-.5) .. (-1,0);
\draw[white, line width = 10] (0,-3) .. controls +(-2, 0) and +(0,-1) .. (-4,0);
\draw (0,-1) .. controls +(-.5, 0) and +(0,-.5) .. (-1,0);
\draw (0,-3) .. controls +(-2, 0) and +(0,-1) .. (-4,0);
\end{scope}
\end{tikzpicture}
    \caption{A diagram of the link $L7a3$ (left), and a diagram for $\widetilde{L}$ (right).}
    \label{fig:doublecover7a3}
\end{figure}

It remains to show that $L$ cannot be split by a self crossing change within its unknotted component $K_2$.
By Proposition~\ref{lem:Nullhomotopic}, $K_2$ must be trivial in $\pi_1 (\mathbb{S}^3\setminus K_1)$, which admits~$\langle x,\, y\,\vert\, yxy = xyx \rangle$ as a Wirtinger presentation.
Here, $x$ and $y$ are the generators depicted in Figure~\ref{fig:gens pi_1}.
With respect to these generators, $K_2$ can be written as $xy^{-1}$ (or~$x^{-1}y$ depending on the orientation).
If $K_2$ were nullhomotopic then $x=y$, and thus~$\pi_1(\mathbb{S}^3\setminus K_1)=\Z$ which is absurd since $K_1$ is a trefoil.
%$ B_3= \bZ$, 
Therefore, $\wsp(L)=2$.
\end{exa}

\section{The weak splitting number of small links}
\label{sec:Table}

Table \ref{tab:small_linksI} below lists $\wsp(L)$ for prime links with $9$ or fewer crossings.
Its second column indicates which of the previously described methods we use among the following:
%In the third column of Table \ref{tab:small_linksI}, we indicate the methods used among the following:
\begin{itemize}
\item[(0)] non-splitness: the link is non-split and has $\wsp(L) \leq 1$;
\item[(1)] the linking number bound from Lemma~\ref{lem:pre_linking_bound};
\item[(2)] the slice-torus or signature bound, for the values of $\om\neq 1$ used see Table \ref{tab:angles}; \footnote{In Table~\ref{tab:angles}, we only list \textit{one} variable $\omega \in S^1$ for the multivariable signature $\sigma_L$: for these links, we are using~$\sigma_L(\omega,\ldots,\omega)$ and its relation to the Levine-Tristram signature at $\omega$; recall~\eqref{eq:CFProposition25}.}
\item[(3)] the Alexander polynomial obstructions from~\cite{BorodzikFriedlPowell};
\item[(4)] the covering link calculus obstruction from Proposition~\ref{cor:Covering};
\item[(5)] the homotopical considerations of Lemmas~\ref{prop:MilnorTriple} and~\ref{lem:Nullhomotopic};
\item[(6)] the combination of covering link calculus and the fundamental group obstruction.
\end{itemize}
\begin{table}
    \centering
    \begin{tabular}{cc|cc|cc}
       Link as in \cite{LinkInfo} &  $\theta$ & Link as in \cite{LinkInfo} &  $\theta$& Link as in \cite{LinkInfo} &  $\theta$\\
       \hline\hline
       $L9a52\{1,0\}$  & $3\pi/97$ &
       $L9n14\{0\}$ & $5\pi/19$&$L9n17\{0\}$& $59\pi/61$\\
        $L9n24\{1,0\}$& $2\pi/13$&$L9n28\{0,0\}$ & $2\pi/17$ && \\
       
    \end{tabular}
    \caption{The roots of unity~$\omega = e^{2i\theta}$ used to compute the signature bound of~Theorem~\ref{thm:SignatureIntro} for the entries marked with $\star$ in Table \ref{tab:small_linksI}.}\label{tab:angles}
\end{table}

\begin{table}[]
\begin{tabular}{ccc}
name & $\wsp$ & method \\
\hline\hline
L2a1 & 1 & (0)\\
L4a1 & 2 & (1)\\
L5a1 & 1 & (0)\\
L6a1 & 2 & (1)\\
L6a2 & 3 & (1)\\
L6a3 & 3 & (1)\\
L6a4 & 2 & (3)\\
L6a5 & 3 & (1)\\
L6n1 & 3 & (1)\\
L7a1 & 1 & (0)\\
L7a2 & 2 & (1)\\
L7a3 & 2 & (6)\\
L7a4 & 2 & (4)\\
L7a5 & 1 & (0)\\
L7a6 & 2 & (2)\\
L7a7 & 3 & (1)\\
L7n1 & 2 & (1)\\
L7n2 & 1 & (0)\\
L8a1 & 1 & (0)\\
L8a2 & 1 & (0)\\
L8a3 & 2 & (1)\\
L8a4 & 1 & (0)\\
L8a5 & 2 & (1)\\
L8a6 & 2 & (1)\\
L8a7 & 2 & (1)\\
L8a8 & 2 & (2)\\
L8a9 & 2 & (3)\\
L8a10 & 3 & (1)\\
L8a11 & 3 & (1)\\
L8a12 & 4 & (1)\\
L8a13 & 4 & (1)\\
L8a14 & 4 & (1)\\
L8a15 & 3 & (1)\\
L8a16 & 3 & (3)\\
L8a17 & 4 & (1)\\
L8a18 & 4 & (1)\\
L8a19 & 2 & (1)\\
L8a20 & 4 & (1)\\
L8a21 & 4 & (1)\\
L8n1  & 2 & (1)\\
L8n2  & 1 & (0)\\
L8n3  & 4 & (1)\\
L8n4  & 4 & (1)\\
L8n5  & 2 & (1)\\
L8n6  & 4 & (1)\\
L8n7  & 4 & (1)\\
L8n8  & 4 & (1)\\
\end{tabular}%\hspace{.01\textwidth}
\begin{tabular}{|ccc}
name & $\wsp$ & method \\
\hline\hline
L9a1  & 1 & (0)\\
L9a2  & 1 & (0)\\
L9a3  & 1 & (0)\\
L9a4  & 2 & (6)\\
L9a5  & 2 & (1)\\
L9a6  & 2 & (1)\\
L9a7  & 2 & (1)\\
L9a8  & 2 & (6)\\
L9a9  & 2 & (4) \\
L9a10 & 2 & (6)\\
L9a11 & 2 & (1)\\
L9a12 & 2 & (1)\\
L9a13 & 2 & (1)\\
L9a14 & 2 & (6)\\
L9a15 & 2 & (1)\\
L9a16 & 2 & (1)\\
L9a17 & 2 & (6)\\
L9a18 & 2 & (6)\\
L9a19 & 2 & (1)\\
L9a20 & 2 & (3)\\
L9a21 & 1 & (0)\\
L9a22 & 2 & (3)\\
L9a23 & 3 & (1)\\
L9a24 & 2 & (3)\\
L9a25 & 2 & (3)\\
L9a26 & 2   & (2)\\
L9a27 & 1   & (0)\\
L9a28 & 3   & (1)\\
L9a29 & 2/3 & (2)\\
L9a30 & 2/3 & (2)\\
L9a31 & 1   & (0)\\
L9a32 & 3   & (1)\\
L9a33 & 3   & (1)\\
L9a34 & 2   & (1)\\
L9a35 & 2   & (3)\\
L9a36 & 3   & (2)\\
L9a37 & 2   & (1)\\
L9a38 & 1   & (0)\\
L9a39 & 2   & (1)\\
L9a40 & 2   & (2)\\
L9a41 & 2   & (1)\\
L9a42 & 2   & (3)\\
L9a43 & 3   & (1)\\
L9a44 & 3   & (1)\\
L9a45 & 3   & (1)\\
L9a46 & 2   & (2)\\
L9a47 & 3   & (3)\\
\end{tabular}%\hspace{.01\textwidth}
\begin{tabular}{|ccc}
name & $\wsp$ & method \\
\hline\hline
L9a48 & 4   & (1)\\
L9a49 & 4   & (1)\\
L9a50 & 3   & (2)\\
L9a51 & 4   & (1)\\
L9a52 & 3   & (2)$^{\star}$\\
L9a53 & 2   & (4)\\
L9a54 & 3   & (2)\&(5)\\
L9a55 & 4   & (1)\\
L9n1 & 2 & (1)\\
L9n2 & 2 & (6)\\
L9n3 & 1 & (0)\\
L9n4 & 2 & (1)\\
L9n5 & 2 & (2)\\
L9n6 & 1 & (0)\\
L9n7 & 2 & (1)\\
L9n8 & 1 & (0)\\
L9n9 & 2 & (1) \\
L9n10 & 2 & (1)\\
L9n11 & 2 & (1)\\
L9n12 & 2 & (1)\\
L9n13 & 1 & (0)\\
L9n14 & 2 & (2)$^{\star }$\\
L9n15 & 3 & (1)\\
L9n16 & 3 &  (1)\\
L9n17 & 2 & (2)$^{\star }$\\
L9n18 & 4 & (1)\\
L9n19 & 4 & (1)\\
L9n20 & 3 & (1)\\
L9n21 & 3 & (1)\\
L9n22 & 3 & (1)\\
L9n23 & 3 & (2)\\
L9n24 & 3 & (2)$^{\star }$\\
L9n25 & 2 & (3)\\
L9n26 & 3 & (2)$^{ }$\\
L9n27 & 1 & (0)\\
L9n28 & 3 & (2)$^{\star }$\\
 &  & \\
 &   & \\
  &  & \\
 &   & \\
  &  & \\
 &   & \\
  &  & \\
 &   & \\
  &  & \\
 &   & \\
  &  & \\
\end{tabular}
\caption{Weak splitting numbers of prime links with $9$ or fewer crossings.
For the entries with a $\star$, the values of $\omega$ used to compute the signature bound are listed in Table~\ref{tab:angles}.
}\label{tab:small_linksI}
\end{table}

\bibliographystyle{plain}
\bibliography{biblioweaksplitting}

\end{document}